\documentclass[reqno]{amsart} 
\usepackage{amsmath,amssymb,amsthm}
\newtheorem{theorem}{Theorem}[section]
\newtheorem{lemma}[theorem]{Lemma}

\newtheorem{corollary}[theorem]{Corollary}
\newtheorem{proposition}[theorem]{Proposition}
\theoremstyle{definition}
\newtheorem{definition}[theorem]{Definition}
\newtheorem{example}[theorem]{Example}

\newtheorem{remark}[theorem]{Remark}

\numberwithin{equation}{section}

\newcommand{\D}{{\mathbb D}}

\newcommand{\cL}{{\mathcal L}}
\newcommand{\cE}{{\mathcal E}}
\newcommand{\cS}{{\mathcal S}}
\newcommand{\cX}{{\mathcal X}}
\newcommand{\cY}{{\mathcal Y}}
\newcommand{\cG}{{\mathcal G}}
\newcommand{\cU}{{\mathcal U}}
\newcommand{\cO}{{\mathcal O}}
\newcommand{\cM}{{\mathcal M}}

\newcommand{\balpha}{{\boldsymbol \alpha}}
\newcommand{\bbeta}{{\boldsymbol \beta}}
\newcommand{\bgamma}{{\boldsymbol \gamma}}
\newcommand{\bS}{{\mathbf S}}

\newcommand{\sbm}[1]{\left[\begin{smallmatrix} #1
                \end{smallmatrix}\right]}

\begin{document}
\title[weighted Hardy multiplier classes]
{Contractive multipliers from Hardy space to weighted Hardy space}

\author{Joseph A. Ball}
\address{Department of Mathematics,
Virginia Tech, Blacksburg, VA 24061-0123, USA}
%\email{joball@math.vt.edu}
\author{Vladimir Bolotnikov}
\address{Department of Mathematics,
The College of William and Mary,
Williamsburg VA 23187-8795, USA}
%\email{vladi@math.wm.edu}


\begin{abstract}
    It is shown how any contractive multiplier from the Hardy space to 
    a weighted Hardy space $H^{2}_{\bbeta}$ can be factored as a fixed factor 
    composed with the classical Schur multiplier (contractive 
    multiplier between Hardy spaces). The result is applied to get 
    results on interpolation for a Hardy-to-weighted-Hardy contractive 
    multiplier class.
\end{abstract}
\maketitle

\section{Introduction}   \label{S:intro}

Given a sequence $\bbeta=\{\beta_j\}_{j\ge 0}$ of positive numbers such that $\beta_0=1$ and 
$\liminf\beta_j^{\frac{1}{j}}\ge 1$, the {\em weighted Hardy space} $H^2_\bbeta$
is defined as the set of all functions analytic on the  open unit 
disk $\D$ and with finite norm 
$\|f\|_{H^2_{\bbeta}}$ given by
$$
\|f\|_{H^2_{\bbeta}}^2=\sum_{j=0}^\infty \beta_j |f_j|^2\quad\mbox{if}\quad
f(z)=\sum_{j=0}^\infty  f_jz^j.
$$
Polynomials are dense in  $H^2_\bbeta$ and  the monomials $\{z^k\}_{k\ge 0}$  form an orthogonal set
uniquely defining the weight sequence $\bbeta$  by $\beta_j=\|z^j\|^2$ for $j\ge 0$. The space $H^2_\bbeta$
can be alternatively characterized as the reproducing kernel Hilbert space with reproducing kernel
\begin{equation}
k_{\bbeta}(z,\overline{\zeta})=\sum_{j=0}^\infty \frac{z^j\overline{\zeta}^j}{\beta_j}.
\label{0.1}
\end{equation}
For a Hilbert (coefficient) space $\cY$, we denote by $H^2_{\bbeta}(\cY)$ the  reproducing kernel 
Hilbert space
with  reproducing kernel $k_{\bbeta}(z,\overline{\zeta})I_{\cY}$ which can be characterized explicitly as 
follows:
\begin{equation}
H^2_{\bbeta}(\cY)=\left\{f(z)={\displaystyle\sum_{k\ge 0}f_kz^k}: \;
\|f\|^2_{H^2_{\bbeta}(\cY)}:={\displaystyle \sum_{k\ge 0}\beta_k \cdot 
\|f_k\|_{\cY}^2}<\infty\right\}.
\label{0.2}
\end{equation}
We will write ${\bf 1}$ rather than $\bbeta$ if $\beta_j=1$ for all $j\ge 0$. Observe that 
$H^2_{\bf 1}(\cY)$ is the classical vector Hardy space $H^2(\cY)$ 
of the unit  disk (with reproducing kernel $k_{\bf 1}(z,\overline{\zeta})=(1-z\overline{\zeta})^{-1}\cdot 
I_{\cY}$); the choice $\beta_j= \frac{j!(n-1)!}{(j+n-1)!}$ yield the standard weighted Bergman space 
$A^2_{n}(\cY)$ ($n\ge 1$) and in particular, the classical Bergman 
space $A^2_{2}(\cY)$ in case $n=2$.  A general reference for such 
spaces and the associated weighted shift operators is the article of 
Shields \cite{Shields}.

\smallskip

For $\cU$ and $\cY$ any pair of Hilbert spaces, we use the notation
$\cL(\cU, \cY)$ to denote the space of bounded, linear operators
from $\cU$ to $\cY$, shortening  
%the notation 
$\cL(\cX, \cX)$ to $\cL(\cX)$.
\begin{definition}
An $\cL(\cU,\cY)$-valued function $S$ defined on $\D$ is called a {\em contractive multiplier} 
from $H^2_{\balpha}(\cU)$ to $H^2_{\bbeta}(\cY)$, denoted as $S \in 
\cS_{\balpha \to \bbeta}(\cU, \cY)$, if the multiplication 
operator  $M_S \colon f\to Sf$ 
is a contraction from $H^2_{\balpha}(\cU)$ to $H^2_{\bbeta}(\cY)$.
\label{D:0.1}
\end{definition}
The  latter means that the operator $I-M_SM_S^*$
%$I_{H^2_{\bbeta}(\cY)}-M_SM_S^*$ 
(with $M_{S}$ considered as acting from the Hardy space $H_\balpha^{2}(\cU)$ 
into $H^{2}_{\bbeta}(\cY)$)  is positive semidefinite; this in turn is equivalent to the kernel
\begin{equation}
K_S(z,\overline{\zeta}):=k_{\bbeta}(z,\overline{\zeta})I_{\cY}-S(z)S(\zeta)^* \cdot 
k_{\balpha}(z, \overline{\zeta})
\label{0.5}
\end{equation}
being positive on $\D\times\D$, denoted symbolically as $K_S(z,\overline{\zeta})\succeq 0$. 
In case $\bbeta=\balpha={\bf 1}$,  the set of contractive multipliers from $H^2(\cU)$ to  
$H^2(\cY)$ is the classical Schur class $\cS(\cU,\cY)$ of functions analytic on $\D$ whose 
values are contractive operators from $\cU$ to $\cY$.  
Our main focus here will be on the 
intermediate case where $\balpha={\bf 1}$, and we will write $\cS_{\bbeta}(\cU, \cY)$ (rather 
than $\cS_{{\bf 1}\to\bbeta}(\cU, \cY)$) for the set of all contractive multipliers from 
$H^2(\cU)$ to $H^2_{\bbeta}(\cY)$.

\smallskip

In Section \ref{S:fact} below, for the case where the sequence 
$\bbeta$ is non-increasing, we show that elements of 
$\cS_{\bbeta}(\cU, \cY)$ can be factored as $S(z) = \Psi_{\bbeta}(z) 
{\mathbf S}(z)$ where $\Psi_{\bbeta}$ is a fixed factor, and where 
${\mathbf S}$ is in a classical Schur class (see Theorem 
\ref{T:HBSchur} below).  This enables us to obtain a complete 
solution of a general  (left-tangential with 
operator argument) interpolation problem in the class 
$\cS_{\bbeta}(\cU, \cY)$ by reducing the problem to a well understood interpolation 
problem for the function $\bS$ in the classical Schur class 
(see Theorems \ref{T:IPexistence} and \ref{T:1.2} below).  
In Section \ref{S:alpha} we give an example to illustrate how the 
results do not generalize to the more general Schur class 
$\cS_{\balpha \to \bbeta}$ in case $\balpha \ne {\mathbf 1}$.

\smallskip

There has been much interest of late in so-called Bergman inner 
functions, i.e., functions which map the  coefficient space $\cU$ 
isometrically onto a shift-invariant subspace ${\mathcal L}$ for the 
shift operator $S_{\bbeta} \colon f(z) \mapsto z f(z)$ on 
$H^{2}_{\bbeta}(\cU)$ (see \cite{HKZ, olieot, olaa}), especially for 
the case where $\beta_{j} = \frac{j! (n-1)!}{(j+n-1)!}$ yields one of 
the standard weighted Bergman spaces.  It is known that Bergman inner 
functions are contractive multipliers from the Hardy space $H^{2}$ to 
the Bergman space $H^{2}_{\bbeta}$. Hence there is a potential for 
the results of this paper to apply to Bergman inner functions as well.

\section{The Hardy-to-weighted Hardy contractive multiplier class}  
\label{S:fact}

Let us assume now that {\em the weighted sequence $\bbeta=\{\beta_j\}_{j\ge 0}$ is 
non-increasing}
%\begin{equation}
%\beta_{j+1}\le \beta_j\quad\mbox{for all} \quad j\ge 0,
%\label{0.5a}\end{equation}
and introduce the positive sequence $\bgamma=\{\gamma_j\}_{j\ge 0}$ by
\begin{equation}
\gamma_0=1,\quad 
\gamma_j=\left(\beta_j^{-1}-\beta_{j-1}^{-1}\right)^{-1}=\frac{\beta_j\beta_{j-1}}{\beta_{j-1}-\beta_j}
\quad (j\ge 1).
\label{defgamma}
\end{equation}
Since the sequence $\bgamma$ is positive, the kernel
\begin{equation}
\widetilde{k}_{\bbeta}(z,\overline{\zeta}):=(1-z\overline{\zeta})\cdot k_{\bbeta}(z,\overline{\zeta})=
\sum_{j=0}^{\infty}\gamma_{j}^{-1}z^j\overline{\zeta}^j=k_\bgamma(z,\overline{\zeta})
\label{0.5b}
\end{equation}
is positive (and analytic in $z$ and conjugate analytic in $\zeta$ 
for $z, \zeta$ in $\D$). We then define the operator-valued 
function  $\Psi_\bbeta: \, \D\to \cL(\ell_2(\cY),\cY)$ by 
\begin{equation}
\Psi_{\bbeta}(z): \; \{y_j\}_{j\ge 0}\;  \mapsto \; y_{0}+\sum_{j=1}^\infty
\sqrt{\beta_j^{-1}-\beta_{j-1}^{-1}} \; \cdot  y_j 
z^j=\sum_{j=0}^\infty \frac{y_j}{\sqrt{\gamma_j}} \, z^j.
\label{2.7}
\end{equation}  
The multiplication operator $M_{\Psi_{\bbeta}}$ is an isometry from $\ell_2(\cY)$
onto the weighted Hardy space $H^2_{\bgamma}(\cY)$ (in fact, $\Psi_{\bbeta}$ is a weighted 
$Z$-transform).
It is convenient to represent $\Psi_{\bbeta}(z)$  and the elements $ {\bf y}=\{y_j\}_{j\ge 
0}$ in $\ell_2(\cY)$ in the matrix  form
$$
\Psi_{\bbeta}(z)=\begin{bmatrix}I_{\cY} & {\displaystyle\frac{z}{\sqrt{\gamma_1}}I_{\cY}} &
 {\displaystyle\frac{z^2}{\sqrt{\gamma_2}}I_{\cY}} & \ldots &\end{bmatrix},\qquad
{\bf y}=\begin{bmatrix} y_0 \\ y_1 \\ \vdots \end{bmatrix}\in\ell_2(\cY).
$$
It is readily seen from \eqref{0.5b} that 
\begin{equation}
\Psi_{\bbeta}(z)\Psi_{\bbeta}(\zeta)^*=\widetilde{k}_{\bbeta}(z,\overline{\zeta})\cdot 
I_{\cY}=(1-z\overline{\zeta})\cdot k_{\bbeta}(z,\overline{\zeta})\cdot I_{\cY}.
\label{2.7a}
\end{equation}
\begin{theorem}  \label{T:HBSchur}
Let the weight sequence $\bbeta$ be non-increasing. The function $S$ is in the class 
 $\cS_\bbeta(\cU, \cY)$ if and only if  there is an $\bS$  in the Schur class 
$\cS(\cU, \ell_{2}(\cY))$ so that 
 \begin{equation}  \label{Schur-fact}
     S(z) = \Psi_{\bbeta}(z) \bS(z).
   \end{equation}
\end{theorem}

\begin{proof}  Suppose first that $S$ has the form 
    \eqref{Schur-fact} with $\bS$ in $\cS(\cU, \ell_{2}(\cY))$.  
    Then we compute, making use of \eqref{2.7a}, that 
    \begin{align*}
k_{\bbeta}(z,\overline{\zeta})\cdot I_{\cY}- \frac{ S(z) 
	S(\zeta)^{*}}{ 1 - z \bar{\zeta}}  & =
k_{\bbeta}(z,\overline{\zeta})\cdot I_{\cY} - \frac{ \Psi_{\bbeta}(z) \bS(z) 
    \bS(\zeta)^{*} \Psi_{\bbeta}(\zeta)^{*}}{(1 - z \bar{\zeta})}  \\
    & = \Psi_{\bbeta}(z) \left[ \frac{I - \bS(z) \bS(\zeta)^{*}}{ 1 - z 
    \bar{\zeta}} \right] \Psi_{\bbeta}(\zeta)^{*}  \succeq 0,
    \end{align*}
    and it follows that $S \in \cS_\bbeta(\cU,\cY)$ by the 
    criterion that the kernel in \eqref{0.5} be positive.
    
\smallskip

    Conversely, suppose that $S$ is in the  class $\cS_\bbeta(\cU,\cY)$.  It then follows that 
    the kernel \eqref{0.5} is positive.  From \eqref{2.7a} we see that
    $$
	K_{S}(z, \overline{\zeta}) = \frac{ \Psi_{\bbeta}(z) \Psi_{\bbeta}(\zeta)^{*} 
	- S(z) S(\zeta)^{*}}{ 1 - z \bar{\zeta}},
	$$
	and hence the right-hand side is a positive kernel.  It then 
	follows from the theorem of Leech \cite[p. 107]{RR} that there is an $\bS$ in 
	the Schur class $\cS(\cU, \ell_{2}(\cY))$ so that 
	$S = \Psi_{\bbeta} \bS$, i.e., \eqref{Schur-fact} holds.
   \end{proof}
   
The representation formula \eqref{Schur-fact} makes it possible to reduce certain questions 
concerning the generalized Schur class $\cS_{{\mathbf 1} \to 
\bbeta}$ to well-understood questions concerning the 
classical Schur class $\cS$.  In the next section we demonstrate how 
this principle can be applied in the context of interpolation.

\section{Multiplier interpolation problems}

In this section we study a Nevanlinna-Pick type interpolation problem 
in the class  $\cS_\bbeta(\cU,\cY)$. To formulate the problem we need several definitions.

\smallskip

A pair $(E,T)$ consisting of operators $T\in\cL(\cX)$ and $E\in\cL(\cX,\cY)$ is called 
an {\em output pair}. An output pair $(E,T)$ is called {\em $\bbeta$-output-stable} if  
the associated $\bbeta$-observability operator
\begin{equation}
\cO_{\bbeta,E,T}: \; x\mapsto E \, k_{\bbeta}(z,T)x=\sum_{j=0}^\infty (\beta_j^{-1}ET^jx) \, z^j
\label{0.6}
\end{equation} 
maps $\cX$ into $H^2_{\bbeta}(\cY)$ and is bounded. If $(E,T)$ is $\bbeta$-output stable, then the 
{\em $\bbeta$-observability gramian}
$$
{\mathcal G}_{\bbeta,E, T}:=(\cO_{\bbeta,E, T})^{*}\cO_{\bbeta,E,T}
$$
is bounded on $\cX$ and can be represented via the series
\begin{equation}
{\mathcal G}_{\bbeta,E, T}= \sum_{k=0}^\infty \beta_j^{-1}\cdot T^{*k}E^*ET^k
\label{0.8}
\end{equation}
converging in the strong operator topology. Observe that in case $\bbeta={\bf 1}$, the 
observability operator \eqref{0.6} amounts to the well-known observability operator
$$
\cO_{{\bf 1},E,T}: \; x\mapsto E(I-zT)^{-1}x
$$
and the ${\bf 1}$-output stability means that $\cO_{{\bf 1},E,T}$ is bounded as an operator 
from $\cX$ to $H^2(\cY)$. 
%A $\bbeta$-output stable pair $(E,T)$ is called 
%{\em observable} if ${\mathcal G}_{\bbeta,E, T}$ has trivial kernel and it is called 
%exactly observable if  ${\mathcal G}_{\bbeta,E, T}$ is bounded below. 

\smallskip

For a $\bbeta$-output stable pair $(E,T)$, we define the tangential functional calculus $f\mapsto
(E^*f)^{\wedge L}(T^*)$ on $H^2_{\bbeta}(\cY)$ by
\begin{equation}
(E^*f)^{\wedge L}(T^*)=\sum_{j=0}^\infty T^{* j}E^{*}f_j\quad\mbox{if}\quad
f(z)=\sum_{j=0}^\infty f_jz^j \in H^2_{\bbeta}(\cY).
\label{0.9}
\end{equation}
The computation
\begin{align*} 
\left\langle \sum_{j=0}^\infty T^{* j}E^{*}f_j, \, x\right\rangle_{\cX}=&
\sum_{j=0}^\infty \left\langle f_j, \, ET^jx\right\rangle_{\cY}\\
=&\sum_{j=0}^\infty \beta_{j}\cdot \left\langle f_j, \,
\beta_j^{-1} ET^jx\right\rangle_{\cY}=
\left\langle f, \, \cO_{\bbeta,E,T}x\right\rangle_{H^2_{\bbeta}(\cY)}
\end{align*}
shows that the $\bbeta$-output stability of $(E, T)$ is exactly
what is needed to verify that the infinite series in the definition
\eqref{0.9} of $(E^*f)^{\wedge L}(T^*)$ converges in the weak topology on $\cX$.
The same computation shows that tangential evaluation
with operator argument amounts to the adjoint of $\cO_{\bbeta,E, T}$:
\begin{equation}   \label{evvecfunc}
(E^*f)^{\wedge L}(T^*)= \cO_{\bbeta,E, T}^{*}f\quad\mbox{for}\quad f \in H^2_{\bbeta}(\cY).
\end{equation}
The evaluation map extends to multipliers $S\in\cS_\bbeta(\cU,\cY)$ by
\begin{equation}
(E^*S)^{\wedge L}(T^*)= \cO_{\bbeta,E, T}^{*}M_S\vert_{\cU}.
\label{0.11}
\end{equation}
The objective of this section is to study the interpolation problem {\bf IP} whose data set 
consists of three operators 
\begin{equation}
T\in\cL(\cX),\quad E\in\cL(\cX,\cY), \quad N\in\cL(\cX,\cU)
\label{0.10}  
\end{equation}
such that the pair $(E,T)$ is $\bbeta$-output stable and the pair  $(N,T)$ is ${\bf 1}$-output stable.

\bigskip

{\bf IP}: {\em Given operators \eqref{0.10}, find all 
contractive multipliers $S\in\cS_\bbeta(\cU,\cY)$ such that 
\begin{equation}
(E^*S)^{\wedge L}(T^*):=\cO_{\bbeta,E, T}^{*}M_S\vert_{\cU}=N^*.
\label{0.12} 
\end{equation}} 

\begin{example}
{\rm By way of motivation of problem {\bf IP}, we note that if we take the data set $(T, E,N)$ of the 
form
$$
T = \begin{bmatrix} \overline{z}_{1}I_{\cY} & & 0\\  & \ddots & \\ 0& & \overline{z}_{k}I_{\cY}
\end{bmatrix}, \quad E = \begin{bmatrix} I_{\cY} & \ldots & I_{\cY}
\end{bmatrix}, \quad N = \begin{bmatrix} V_{1}^* & \ldots & V_{k}^*\end{bmatrix} 
$$
for some $z_{1}, \dots, z_{k}\in\D$ and $V_1,\ldots,V_k\in\cL(\cU,\cY)$,
then it follows from \eqref{0.9} that  
$$
(E^*S)^{\wedge L}(T^*)=\begin{bmatrix} S(z_1) \\ \vdots \\ S(z_k)\end{bmatrix},
$$
so that condition \eqref{0.12} transcribes to Nevanlinna-Pick interpolation conditions
\begin{equation}
S(z_{i}) =V_i \quad\text{for}\quad i = 1, \dots, k.
\label{NPcond} 
\end{equation}}
\label{E:2.3}
\end{example}

The stability assumption for the pair $(E,T)$ is needed to define the expression on the left side of 
\eqref{0.12}. We next show that the stability assumption for the pair $(N,T)$ is not restrictive.

\begin{proposition}
Let us assume that the pair $(E,T)$ is $\bbeta$-output stable and that 
there is a function $S\in\cS_\bbeta(\cU,\cY)$ satisfying condition \eqref{0.12}. Then 
the pair $(N,T)$ is ${\bf 1}$-output stable and the following equality holds:
\begin{equation}
\cO_{\bbeta,E, T}^*M_S=\cO_{{\bf 1},N,T}^*: \; H^2(\cU)\to \cX.
\label{1.1u}
\end{equation}
Furthermore, the observability gramian 
\begin{equation}
\cG_{{\bf 1},N,T}:=\cO_{{\bf 1},N,T}^*\cO_{{\bf 1},N,T}=\sum_{j=0}^\infty T^{*j}N^*NT^j
\label{1.1v}  
\end{equation}
satisfies the Stein identity
\begin{equation}
\cG_{{\bf 1},N,T}-T^*\cG_{{\bf 1},N,T}T =N^*N.
\label{1.8}
\end{equation}
\label{P:neces}
\end{proposition}

\begin{proof} Let $S$ be a function in  $S\in\cS_\bbeta(\cU,\cY)$ subject to \eqref{0.12}.
Then for a function $h(z)={\displaystyle\sum_{j=0}^\infty h_jz^j}\in 
H^2(\cU)$, we have $(M_{S}h)(z) = {\displaystyle\sum_{\ell=0}^{\infty} \left( 
\sum_{j = 0}^{k} S_{j} h_{\ell-j} \right) z^{\ell}}$.  Hence, as a 
consequence of \eqref{evvecfunc} we have
$$
\cO^*_{\bbeta,E,T}M_Sh= \left( E^{*}(Sh) \right)^{\wedge L}(T^{*}) = 
\sum_{\ell=0}^{\infty} T^{*\ell} E^{*}\left( \sum_{j = 0}^{\ell}S_{j} 
h_{\ell-j}\right),
$$
where the latter series converges weakly since the pair $(E,T)$ is $\bbeta$-output stable
and since $Sh\in H^2_\bbeta(\cY)$. If we regularize the series by 
replacing $S_{j}$ by $r^{j} S_{j}$ and replacing 
$h_{i}$ by $r^{i} h_{i}$, we even get that the 
double series in \eqref{1.3u}, after taking the inner product against 
a fixed vector $x \in \cX$, converges absolutely.  We may then 
rearrange the series to have the form
$$
\langle \cO^*_{\bbeta,E,T}M_{S_{r}}h_{r}, x \rangle = \sum_{j,k=0}^{\infty} 
\langle r^{j+k}(T^{*})^{j+k} E^{*} S_{j} h_{k}, \; x \rangle.
$$
We may then invoke Abel's theorem to take the limit as $r \uparrow 1$ (justified by the facts 
that $(E,T)$ is $\bbeta$-output stable and that $Sh \in 
H^{2}_\bbeta(\cY)$---see \cite[page 175]{Rudin}) to get
\begin{equation}   \label{1.3u}
 \cO_{\bbeta,E,T}^{*} M_{S}h = (E^{*}Sh)^{\wedge L}(T^{*}) = 
 \sum_{j,k=0}^{\infty} (T^{*})^{j+k} E^{*} S_{j} h_{k}.
\end{equation}
On the other hand, due to \eqref{0.12},
\begin{align*}
\cO_{{\bf 1},N,T}^*h=&(N^*h)^{\wedge L}(T^*)
=\sum_{k=0}^\infty T^{*k}N^*h_k=\sum_{k=0}^\infty T^{*k}\cO^*_{\bbeta,E,T}Sh_k\\
=&\sum_{k=0}^\infty T^{* k}
\left(\sum_{j=0}^\infty T^{*j}E^*S_j\right)h_k
=\sum_{j,k=0}^\infty (T^*)^{j+k}E^*S_j h_k
\end{align*}
where all the series converge weakly, since that in \eqref{1.3u} does.
Since $h$ was picked arbitrarily in $H^2(\cU)$, the last equality and
\eqref{1.3u} imply \eqref{1.1u}. Therefore, the operator $\cO_{{\bf 1},N, T}^*: \,  H^2(\cU)\to \cX$ is 
bounded and hence the pair $(N,T)$ is ${\bf 1}$-output stable. Therefore, the series in \eqref{1.1v}
converges strongly and (as is well known) satisfies the Stein identity \eqref{1.8}. 
\end{proof}

We shall have need of the auxiliary observation operator described in 
the following lemma.

\begin{lemma}
Let us assume that the weight sequence $\bbeta$ is non-increasing and that the pair $(E,T)$ is 
$\bbeta$-output stable. Then the operator  
\begin{equation}
\widetilde{\cO}_{\bbeta,E,T}: \; x\mapsto \left\{\frac{1}{\sqrt{\gamma_j}} \, ET^jx\right\}_{j\ge 0}
\label{1.6}
\end{equation}
(where $\gamma_j$'s are defined from $\bbeta$ as in \eqref{defgamma}) maps $\cX$ into $ \ell^2(\cY)$.
Furthermore, the pair $(\widetilde{\cO}_{\bbeta,E,T}, \, T)$ is ${\bf 1}$-output stable and the following 
relations 
hold: 
\begin{equation}
\widetilde{\cO}_{\bbeta,E,T}^*\widetilde{\cO}_{\bbeta,E,T}=
\cG_{\bbeta,E,T}-T^*\cG_{\bbeta,E,T}T\quad\mbox{and}\quad \cG_{{\bf 1}, 
\widetilde{\cO}_{\bbeta,E,T},T}=\cG_{\bbeta,E,T}.
\label{1.7}
\end{equation}
\label{L:3.1}
\end{lemma}

\begin{proof}
Making use of the power series representation \eqref{0.8} for $\cG_{\bbeta,E,T}$ and the formulas 
\eqref{defgamma} for $\gamma_j$, we get
\begin{align}
\cG_{\bbeta,E,T}-T^*\cG_{\bbeta,E,T}T&=\sum_{k=0}^\infty \beta_k^{-1}\cdot
 T^{*k}E^*ET^k-\sum_{k=1}^\infty \beta_{k-1}^{-1}\cdot T^{*k}E^*ET^k\notag\\
&=E^*E+\sum_{k=1}^\infty \left(\beta_k^{-1}-\beta_{k-1}^{-1}\right) T^{*k}E^*ET^k\notag\\
&=\sum_{k=0}^\infty \gamma_k^{-1}\cdot T^{*k}E^*ET^k=\cG_{\bgamma,E,T}
\label{1.5}
\end{align}
from which we conclude that the series on the right side of \eqref{1.5} converges strongly.
Then we see from \eqref{1.6} that 
$$
\|\widetilde{\cO}_{\bbeta,E,T}x\|^2_{\ell^2(\cY)}=\langle \cG_{\bgamma,E,T}x, \, 
x\rangle_\cX<\infty,
$$
We conclude that $\widetilde{\cO}_{\bbeta,E,T}x$ belongs to $\ell^2(\cY)$ for 
any $x\in\cX$ and that 
$\widetilde{\cO}_{\bbeta,E,T}^*\widetilde{\cO}_{\bbeta,E,T}=\cG_{\bgamma,E,T}$.
Substituting this last relation into \eqref{1.5} gives the first relation in \eqref{1.7}. Finally,  from 
\eqref{defgamma} we see that $\; {\displaystyle \sum_{k=0}^j\gamma_k^{-1}=\beta_j^{-1}}$
for all $ j\ge 0$.
%$$ \sum_{k=0}^j\gamma_k^{-1}=\beta_j^{-1}\quad \text{for all}\quad j\ge 0.$$
We therefore have
\begin{align*}
\cG_{{\bf 1}, \widetilde{\cO}_{\bbeta,E,T},T}:= &\sum_{j=0}^\infty 
T^{*k}\widetilde{\cO}_{\bbeta,E,T}^*\widetilde{\cO}_{\bbeta,E,T}T^k
=\sum_{j=0}^\infty T^{*k}\cG_{\bgamma,E,T}T^k \\
=& \sum_{j=0}^\infty \sum_{k=0}^\infty\gamma_k^{-1}\cdot T^{*(k+j)}E^*ET^{k+j}
= \sum_{j=0}^\infty \sum_{k=0}^j \gamma_k^{-1}\cdot T^{*j}E^*ET^{j}\\
=&\sum_{j=0}^\infty \beta_j^{-1}\cdot T^{*j}E^*ET^{j}=\cG_{\bbeta,E,T},
\end{align*}
and the second equality in \eqref{1.7} follows.  This in turn implies in particular that the pair  
$(\widetilde{\cO}_{\bbeta,E,T}, \, T)$ is ${\bf 1}$-output stable.\end{proof}

We next show how the auxiliary observation operator constructed in 
Lemma \ref{L:3.1} can be used to reduce the problem {\bf IP} to a 
well understood problem for a classical Schur-class function.

\begin{lemma}  \label{P:IPreduction}
Let $\bbeta$ be a non-increasing weight sequence and let $(E,T)$ be a $\bbeta$-output stable pair.
Suppose that $S \in \cS_\bbeta(\cU,\cY)$
is presented  in the form $S = \Psi_{\bbeta} \bS$ with $\bS \in \cS(\cU, \ell_{2}(\cY))$ as in
Lemma \ref{T:HBSchur}. Then 
\begin{equation}   \label{EST*=bigST*}
\left( E^{*}S \right)^{\wedge L}(T^{*})=
\left( E^{*}( \Psi_{\bbeta} \bS) \right)^{\wedge L}(T^{*}) =
\left( (\widetilde \cO_{\bbeta,E,T})^{*} \bS \right)^{\wedge L}(T^{*}).
  \end{equation}
 \end{lemma}
 
 \begin{proof}
     Write out $\bS(z) \colon \cU \to \ell_{2}(\cY)$ as a column
 $$
\bS(z) = \begin{bmatrix} \bS_{1}(z) \\ \bS_{2}(z) \\ \vdots
     \end{bmatrix}\quad\mbox{where}\quad \bS_{j}(z) = \sum_{k=0}^{\infty} 
\bS_{j,k} z^{k}\quad\mbox{with}\quad
{\bf S}_{j,k}\in\cL(\cU,\cY).
 $$
Then $S(z) = \Psi_{\bbeta}(z) \bS(z)$ is given explicitly as
$$
\Psi_{\bbeta}(z) \bS(z)
 = \sum_{j=0}^{\infty} \gamma_j^{-\frac{1}{2}} \bS_{j}(z)  z^{j}    
= \sum_{j=0}^{\infty} \sum_{k=0}^{\infty} \gamma_j^{-\frac{1}{2}}
     \bS_{j,k} z^{j+k} = \sum_{\ell=0}^{\infty} \sum_{k=0}^{\ell}
     \gamma_k^{-\frac{1}{2}}\bS_{k, \ell -k}
     z^{\ell}
$$
where the rearrangement of the infinite series can be justified much 
as in the proof of Proposition \ref{P:neces}.
We conclude that
 \begin{equation}   \label{EST*}
     \left( E^{*} (\Psi_\bbeta \bS) \right)^{\wedge L}(T^{*}) =
     \sum_{\ell=0}^{\infty} \sum_{k = 0}^{\ell}  
     \gamma_k^{-\frac{1}{2}} T^{* \ell}E^* \bS_{k, \ell - k}.
 \end{equation}
 
 On the other hand,
$$
     \left( \widetilde \cO_{\bbeta,E,T}\right)^{*} \bS(z) = 
     \sum_{j=0}^{\infty} \gamma_j^{-\frac{1}{2}} T^{*j} E^{*} \bS_{j}(z)  
= \sum_{j=0}^{\infty} \sum_{k=0}^{\infty}\gamma_j^{-\frac{1}{2}} T^{*j} E^{*} \bS_{j,k} z^{k}
$$
 and hence
$$ 
     \left( (\widetilde \cO_{\bbeta,E,T})^{*} \bS \right)^{\wedge L}
     (T^{*}) = \sum_{j=0}^{\infty} \sum_{k=0}^{\infty}\gamma_j^{-\frac{1}{2}}
     T^{* j+k} E^{*} \bS_{j,k}  
     = \sum_{\ell=0}^{\infty} \sum_{k = 0}^{\ell}\gamma_k^{-\frac{1}{2}}
     T^{*\ell}  E^{*} \bS_{k, \ell-k}.
 $$
Comparison of the latter equality and \eqref{EST*} now gives 
\eqref{EST*=bigST*}.
  \end{proof}
The following consequence of Lemma \ref{P:IPreduction} is immediate.
\begin{corollary}
A function $S$ belongs to the class $\cS_{\bbeta}(\cU,\cY)$ and satisfies interpolation condition \eqref{0.12} 
if and 
only if it is of the form \eqref{Schur-fact} with a Schur-class function   $\bS \in \cS(\cU, \ell_{2}(\cY))$
subject to interpolation condition 
\begin{equation}   \label{0.12'}
    ((\widetilde \cO_{\bbeta,E,T})^{*} \bS)^{\wedge L}(T^{*}) = N^{*}.
\end{equation}
\label{C:3.4}
\end{corollary}
  Applying the known theory for the left-tangential interpolation 
  problem with operator argument for the classical Schur class now 
  leads to the following result.
  
  \begin{theorem} \label{T:IPexistence}
The problem {\bf IP} with the data set \eqref{0.10} 
and the associated observability gramians $\cG_{\bbeta,E,T}$ and $\cG_{{\bf 1},N,T}$
given in \eqref{0.8}, \eqref{1.1v} has a solution  if and only if the associated Pick matrix
      \begin{equation}   \label{Pick}
	  P: = \cG_{\bbeta,E,T} - \cG_{{\bf 1},N,T}
  \end{equation}
  is positive semidefinite. 
 \end{theorem}
 
 \begin{proof}  By Corollary \ref{C:3.4}, solutions $S$ of {\bf IP} 
     exist if and only if the problem \eqref{0.12'} has a solution  $\bS$ in the Schur class
    $\cS(\cU, \ell_{2}(\cY))$. Since the pairs $(\cO_{\bbeta,E,T}, 
    T)$ and $(N,T)$ are both 
${\bf 1}$-output stable,  by the general theory of left-tangential 
     operator-argument Schur-class interpolation 
     (see e.g.~Theorem 4.4 in \cite{bbieot}), we know that the latter holds if and only if the associated Pick 
     matrix is positive semidefinite: 
\begin{equation}   \label{Pick1}
  {\mathbf P}: =   \sum_{j=0}^{\infty} T^{*j}\left[ (\widetilde \cO_{\bbeta,E,T})^{*} 
     \widetilde \cO_{\bbeta,E,T}  - N^{*} N \right] T^{j} \ge 0.
\end{equation}
It is readily seen from \eqref{0.8} and \eqref{1.1v} that 
${\mathbf P} = \cG_{\bbeta,E,T} - \cG_{{\bf 1}, N,T}$ is as in \eqref{Pick}.
      \end{proof}
 
      \smallskip
      
      \textbf{Example 2.1 continued:}
In the case of the Nevanlinna-Pick problem from Example \ref{E:2.3}, we have 
$\cG_{\bbeta,E,T}=\left[k_{\bbeta}(z_i,\overline{z}_j)I_{\cY}\right]_{i,j=1}^k$ and 
$\cG_{{\bf 1},N,T}=\left[k_{\bf 1}(z_i,\overline{z}_j)V_iV_j^*\right]_{i,j=1}^k$
and we conclude from Theorem \ref{T:IPexistence} that {\em there is a contractive multiplier 
$S\in\cS_{\bbeta}(\cU,\cY)$ satisfying the interpolation conditions \eqref{NPcond} if and only if 
the following block-operator is positive semidefinite:}
$$
P=\left[ k_{\bbeta}(z_i,\overline{z}_j)I_{\cY}-\frac{V_iV_j^*}{1-z_i\overline{z}_j}\right]_{i,j=1}^k\ge 0.
$$

\begin{remark}   \label{R:TV} 
     Let $S_\bbeta$ denote the shift operator on $H^2_\bbeta(\cY)$ defined as
$S_\bbeta: \, f(z)\to zf(z)$. It follows that $\|S_\bbeta\|={\displaystyle \sup_{j\ge
0}\frac{\beta_{j+1}}{\beta_j}}$ and thus $S_\bbeta$ is a contraction if and only if the weight sequence 
$\bbeta$
is non-increasing. Reproducing kernel calculations show that its adjoint $S_\bbeta^*$ is given by
\begin{equation}
S^*_\bbeta f=\sum_{k=0}^\infty\frac{\beta_{k+1}}{\beta_k}\cdot f_{k+1}z^k\quad\mbox{if}\quad
f(z)=\sum_{k=0}^\infty f_kz^k.
\label{5.0}
\end{equation}
If $T\in\cL(\cX)$ is strongly stable and the pair $(E,T)$ is $\bbeta$-output stable, then
the range space ${\rm Ran}\, \cO_{\bbeta,E,T}$ with lifted norm is an $S_\bbeta$-invariant (closed) subspace 
of $H^2_\bbeta(\cY)$. Indeed, making use of power series expansion \eqref{0.6} and of
\eqref{5.0} we get
$$
S_\bbeta^*{\mathcal O}_{\bbeta,E, T}x=S_\bbeta^*\sum_{k=0}^\infty \beta_k^{-1} (ET^{k}x)z^k
=\sum_{k=0}^\infty \beta_{k}^{-1}(ET^{k+1}x)z^k={\mathcal O}_{\bbeta,E, T}Tx
$$
from which the desired invariance follows. For  a strongly stable $T\in\cL(\cX)$ and operators 
$E\in\cL(\cY,\cX)$ and 
$N\in\cL(\cX,\cU)$ so that the pairs $(E,T)$ and $(N,T)$ are respectively, $\bbeta$-stable and ${\bf 
1}$-stable, 
define two subspaces 
$$
\cM_1={\rm Ran} \, \cO_{\bbeta,E,T}\subset H^2_\bbeta(\cY)\quad\mbox{and}\quad 
\cM_2={\rm Ran} \, \cO_{{\bf 1},N,T}\subset H^2(\cU)
$$
which are invariant under $S^*_\bbeta$ and $S^*_{\bf 1}$ respectively. Let us define the operator 
$\Phi: \, \cM_1\to\cM_2$ by
\begin{equation}
\Phi: \,  \cO_{\bbeta,E,T}x\to \cO_{{\bf 1},N,T}x\quad\mbox{for all}\quad x\in\cX.
\label{5.00}
\end{equation}
From the formulation \eqref{1.1u} of the interpolation condition 
\eqref{0.12}, it is clear that a necessary condition for the problem 
\textbf{IP} to have a solution is that  $\|\Phi\| \le 1$, or 
equivalently, that
\begin{equation}   \label{Pick'}
P:=\cG_{\bbeta,E,T}- \cG_{{\bf 1},N,T}\ge 0.
\end{equation}
Furthermore, the computation  (where $y = \cO_{\bbeta,E,T}x\in\cM_1$)
\begin{align*}
\Phi S_\bbeta^*y=\Phi S_\bbeta^*\cO_{\bbeta,E,T}x&=\Phi \cO_{\bbeta,E,T}Tx\\
&=\cO_{{\bf 1},N,T}Tx=S_{\bf 1}^*
\cO_{{\bf 1},N,T}x=S_{\bf 1}^*\Phi \cO_{\bbeta,E,T}x=S_{\bf 1}^*\Phi y 
\end{align*}
shows that $\Phi$ intertwines $S_\bbeta^*$ and $S_{\bf 1}^*$.
Since $S_\bbeta$ is a contraction and $S_{\bf 1}$ is an isometry, it follows from the 
Treil-Volberg commutant lifting result \cite{TV} that $\Phi$ can be extended to an operator $R:
H^2_\bbeta(\cY)\to H^2$ such that $\|R\|=\|\Phi\|\le 1$ and $RS_\bbeta^*=S_{\bf 1}^*\Phi$. Its adjoint 
$R^*$ necessarily is the operator of multiplication by a function $S\in\cS_\bbeta(\cU,\cY)$. 
In this way we see that the condition  \eqref{Pick'} is necessary and 
sufficient for the existence of solutions of the problem \textbf{IP}, i.e., we arrive at an 
alternative proof of Theorem \ref{T:IPexistence}.
One of the contributions of the present paper is to obtain an 
explicit description of the set of all solutions of the problem. 
\textbf{IP}.
 \end{remark}

It is known how to parametrize solutions of a left-tangential 
operator-argument interpolation problem for Schur-class functions; 
 the description is more transparent in the case where ${\bf P}(=P)$ is invertible.
Let $J$ be
the operator given by
\begin{equation}
J=\left[\begin{array}{cr}I_{\ell_2(\cY)} &0\\ 0& -I_{\cU}\end{array}\right]\quad\mbox{and
let}\quad \Theta(z)=\begin{bmatrix}A(z) & B(z) \\ C(z) & D(z)\end{bmatrix}
\label{2.3}
\end{equation}
be an $\cL(\ell_2(\cY)\oplus\cU)$-valued function such that for all $z,\zeta\in\D$,
\begin{equation}
\frac{J - \Theta(z)J\Theta(\zeta)^{*}}{1 -z\overline{\zeta}}
=\begin{bmatrix}\widetilde{\cO}_{\bbeta,E,T} \\ N\end{bmatrix}(I - zT)^{-1} P^{-1}(I -
\overline{\zeta}T^{*})^{-1}\begin{bmatrix}\widetilde{\cO}_{\bbeta,E,T}^* & N^*\end{bmatrix}.
\label{2.4}
\end{equation}
The function $\Theta$ is determined by equality \eqref{2.4} uniquely up
to a constant $J$-unitary factor on the right. One possible choice of $\Theta$
satisfying \eqref{2.4} is
$$
\Theta(z)=\begin{bmatrix}D_1 \\ D_2\end{bmatrix}+z\begin{bmatrix}\widetilde{\cO}_{\bbeta,E,T} \\ 
N\end{bmatrix}(I-zT)^{-1}R
$$
where the operator 
$$
\begin{bmatrix}R \\ D_1\\ D_2\end{bmatrix}\colon \ell_2(\cY)\oplus\cU\to
\begin{bmatrix}\cX \\ \ell_2(\cY) \\ \cU\end{bmatrix}
$$ 
is an injective solution to the $J$-Cholesky factorization problem
$$
\begin{bmatrix} R \\ D_1 \\ D_2\end{bmatrix}J
\begin{bmatrix} R^{*}&  D_1^{*}& D_2^* \end{bmatrix} =
\begin{bmatrix} P^{-1} & 0 \\ 0 & J
\end{bmatrix} - \begin{bmatrix} T \\ \widetilde{\cO}_{\bbeta,E,T} \\ N
\end{bmatrix} P^{-1} \begin{bmatrix} T^{*} & 
\widetilde{\cO}_{\bbeta,E,T}^* & N^*\end{bmatrix}.
$$
Such a solution exists due to the identity 
\begin{equation}
P-T^*PT=\widetilde{\cO}_{\bbeta,E,T}^*\widetilde{\cO}_{\bbeta,E,T}-N^*N
\label{1.9}
\end{equation}
which in turn, follows from \eqref{Pick}, \eqref{1.8} and \eqref{1.7}.
If ${\rm spec}(T)\cap{\mathbb
T}\neq{\mathbb T}$ (which is the case if, e.g., $\dim \cX<\infty$), then 
a function $\Theta$ satisfying \eqref{2.4} can be taken in the form
\begin{align}
\Theta(z)=&I+(z-\mu)\left[\begin{array}{c}\widetilde{\cO}_{\bbeta,E,T} \\
N\end{array}\right]
(I-zT)^{-1}P^{-1}\notag\\
&\qquad\qquad\qquad \times(\mu I-T^*)^{-1}\left[\begin{array}{cc}
\widetilde{\cO}^*_{\bbeta,E,T} & -N^*\end{array}\right],\label{2.5}
\end{align}
where $\mu$ is an arbitrary point in ${\mathbb T}\setminus {\rm spec}(T^*)$.
For $\Theta$ of the form \eqref{2.5}, the verification of identity
\eqref{2.4} is straightforward and relies on the Stein identity
\eqref{1.9} only. 

\begin{theorem}
Let us assume that the data set \eqref{0.10} of the problem {\bf IP} is such that
the operator $P$ defined in \eqref{Pick} is strictly positive definite.
Let $\Theta=\sbm{A & B \\ C & D}$ be a function satisfying
\eqref{2.4} and let $\Psi_{\bbeta}$ be given as in \eqref{2.7}. Then 
all solutions $S$ of problem {\bf IP} are parametrized by the formula
\begin{equation}
S=\Psi_{\bbeta}(A\cE+B)(C\cE+D)^{-1}.
\label{2.13}
\end{equation}
where $\cE$ is a free parameter from the Schur class $\cS(\cU,\ell_2(\cY))$.
\label{T:1.2}
\end{theorem}

\begin{proof} Since the pair $(E,T)$ is $\bbeta$-output stable, it follows from Lemma \ref{L:3.1}
that the pair $(\widetilde{\cO}_{\bbeta,E,T},T)$ is ${\bf 1}$-output stable. 
By the theory for left-tangential  operator-argument interpolation in 
the Schur class (see e.g.~\cite{bbieot}), it is known 
that in case the operator \eqref{Pick1} (which is the same as $P$) is strictly positive definite, then 
all functions $\bS\in\cS(\cU, \ell_{2}(\cY))$ satisfying condition \eqref{0.12'}
are given by the formula
 $$
 \bS = (A \cE + B) ( C \cE + D)^{-1}
 $$
 where $\cE$ is a free parameter from the Schur class $\cS(\cU, 
 \ell_{2}(\cY))$.  The Theorem now follows as a consequence of 
 Lemma \ref{P:IPreduction}.
 \end{proof}
 
\begin{remark} \label{R:3.10}
If $\Theta$ is taken in the form \eqref{2.5} and the free parameter 
function $\cE \in \cS(\cU, \ell_{2}(\cY))$ is taken to have the form
$\cE(z) = \left[ \begin{smallmatrix} \cE_{0}(z) \\ \cE_{1}(z) \\ 
\vdots \end{smallmatrix} \right]$ where each $\cE$ is in the Schur 
class $\cS(\cU, \cY)$ subject to $\sum_{j=0}^{\infty} \cE_{j}(z)^{*} 
\cE_{j}(z) \le I_{\cU}$ for $z \in {\mathbb D}$,  then the parametrization 
formula \eqref{2.13} can be written more explicitly as 
\begin{align}
S(z)=&\left(\sum_{j=0}^\infty {\displaystyle\frac{ z^j \cE_j(z)}{\sqrt{\gamma_j}}}
+(z-\mu)Ek_\bbeta(z,T)P^{-1}(\mu
I-T^*)^{-1}{\bf R}_\cE(z)\right)\notag\\
&\quad \times \left(1+ (z-\mu)N(I-zT)^{-1}P^{-1}(\mu
I-T^*)^{-1}{\bf R}_\cE(z)\right)^{-1},\label{parametr}
\end{align}
where $\gamma_j$'s are given as in \eqref{defgamma} and where we have set for short
$$
{\bf R}_\cE(z):=\sum_{j=0}^\infty{\displaystyle\frac{1}{\sqrt{\gamma_j}}}
T^{*j}E^*\cE_j(z)-N^*.
$$
To derive \eqref{parametr}, it is enough to observe that 
$$
\Psi_{\bbeta}(z)\cE(z)=\sum_{j=0}^\infty {\displaystyle\frac{ z^j \cE_j(z)}{\sqrt{\gamma_j}}},\qquad
\widetilde{\cO}^*_{\bbeta,E,T}\cE(z)=\sum_{j=0}^\infty{\displaystyle\frac{1}{\sqrt{\gamma_j}}}
T^{*j}E^*\cE_j(z),
$$
and that on account of \eqref{2.7}, \eqref{1.6}, and \eqref{0.5b},
$$
\Psi_{\bbeta}(z)\widetilde{\cO}_{\bbeta,E,T}(z) =\sum_{j=0}^\infty \gamma_j^{-1}ET^j z^j=
E \, k_{\bgamma}(z,T)=Ek_{\bbeta}(z,T)\cdot(I-zT),
$$
so that 
$$
\Psi_{\bbeta}(z)\widetilde{\cO}_{\bbeta,E,T}(z) (I-zT)^{-1}=Ek_{\bbeta}(z,T).
$$
Substituting \eqref{2.5} into \eqref{2.13} and taking into account the latter expressions
we arrive at \eqref{parametr}. 

\smallskip

If we choose $\cE_0$ to be an arbitrary function in $\cS(\cU,\cY)$ and $\cE_j\equiv 0$ for 
$j\ge 1$, we get a family of solutions to the problem {\bf IP} given by the formula 
\begin{align*}
S(z)=&\left(\cE_0(z)
+(z-\mu)Ek_\bbeta(z,T)P^{-1}(\mu
I-T^*)^{-1}\left(E^*\cE_0(z)-N^*\right)\right)\notag\\ 
&\quad \times \left(1+ (z-\mu)N(I-zT)^{-1}P^{-1}(\mu
I-T^*)^{-1}\left(E^*\cE_0(z)-N^*\right)\right)^{-1}.
\end{align*}
\end{remark}
We now illustrate Theorem \ref{T:1.2} and Remark \ref{R:3.10} by a simple example.

\begin{example}  \label{E:3.2}
For a fixed integer $n\ge 1$, let $\bbeta=\left\{\frac{j!(n-1)!}{(n+j-1)!}\right\}_{j\ge 0}$
and let us write $k_n$ (rather than $k_\bbeta$) for the associated kernel \eqref{0.1}.
We thus have 
$$
k_n(z,\overline{\zeta})=\sum_{j=0}^\infty 
\binom{n+j-1}{j}z^j\overline{\zeta}^j=\frac{1}{(1-z\overline{\zeta})^n}
$$
and the associated reproducing kernel Hilbert space  $H^2_{\bbeta}$ coincides with
the standard weighted Bergman space $A^2_n$. Let us find a contractive multiplier
$S$ from $H^2$ to $A^2_n$ satisfying a single interpolation condition:
\begin{equation}
S(3/4)=4/3.
\label{4.1}   
\end{equation}
We have $T=3/4$, $E=1$, $N=4/3$ and consequently,
$$
P_n:=\cG_{n,E,T}-\cG_{1,N,T}=\left(\frac{16}{7}\right)^n-\frac{256}{63}>0\quad \mbox{for every}\quad n\ge 2.
$$
Therefore, the problem \eqref{4.1} has solutions for every $n\ge 2$, which can be described in   
terms of the linear fractional transformation \eqref{2.13} with free parameter
$$
\cE(z)=\{\cE_{k}(z)\}_{k\ge 0}\quad\mbox{such that}\quad {\displaystyle
\sum_{k=0}^\infty|\cE(z)|^2\le 1}\quad\mbox{for all $z\in\D$}.
$$
The numbers $\gamma_j$ defined in \eqref{defgamma} take the form 
$$
\gamma_j=\frac{1}{\binom{n+j-1}{j}-\binom{n+j-2}{j-1}}=\frac{1}{\binom{n+j-2}{j}}
$$
and the kernel $\widetilde k_{\bbeta} = \widetilde k_{n}$ given by 
\eqref{0.5b} becomes simply
$$
 \widetilde k_{\bbeta}(z, \overline{\zeta}) = \widetilde k_{n}(z, \overline{\zeta})  
= (1 -  z \zeta)^{-(n-1)} = k_{n-1}(z, \overline{\zeta}).
 $$
We can choose $\mu=1$ in formula \eqref{parametr} to get
$$
S(z)=\frac{{\displaystyle\sum_{k=0}^\infty\sqrt{\binom{n+k-2}{k}}
\left(z^k+\frac{4(z-1)}{(1-\frac{3}{4}z)^n}P_n^{-1}
\left(\frac{3}{4}\right)^k\right)\cE_k(z)}-
{\displaystyle\frac{16(z-1)}{3(1-\frac{3}{4}z)^n}P_n^{-1}}}
{{\displaystyle\frac{16(z-1)}{3(1-\frac{3}{4}z)}P_n^{-1}\sum_{k=0}^\infty
\sqrt{\binom{n+k-2}{k}}
\left(\frac{3}{4}\right)^k\cE_{k}(z)}-{\displaystyle\frac{64(z-1)}{9(1-\frac{3}{4}z)}P_n^{-1}}+1}.
$$
To get a particular solution $S$ of the problem \eqref{4.1} we may let $\cE=0$ to arrive at 
$$
S(z)=\frac{7^n(1-z)}{(1-\frac{3}{4}z)(3\cdot 16^{n-1}+4\cdot 7^{n-1}-\left(\frac{9}{4}\cdot 16^{n-1}
+\frac{16}{3}\cdot 7^{n-1}\right)z)}.
$$
\end{example}

\begin{remark} \label{R:5.1}
{\rm In case the Pick operator \eqref{Pick} is positive semidefinite but not invertible,
the solution set of the problem {\bf IP} can be still parametrized by a Redheffer-type formula
\begin{equation}
S=\Psi_{\bbeta}\cdot\left(\widetilde{D}+\widetilde{C}\cE(I-\widetilde{A}\cE)^{-1}\widetilde{B}\right)
\label{5.1}
\end{equation}
where $\cE$ and $\begin{bmatrix}\widetilde{A} & \widetilde{B} \\ \widetilde{C} &
\widetilde{D}\end{bmatrix}$ are Schur-class functions with the coefficient spaces depending
on the degeneracy of the Pick operator $P$. Parametrization
\eqref{5.1} follows from the factorization result
in Lemma \ref{T:HBSchur} and known results on the degenerate interpolation problem \eqref{0.12'} for
Schur-class
functions (see \cite{kh, kheyu}).}
\end{remark}

\section{Contractive multipliers between $H^{2}_{\balpha}$ and 
$H^{2}_{\bbeta}$ in case $\balpha \ne {\mathbf 1}$}   \label{S:alpha}

It is natural to consider contractive multipliers between two weighted Hardy spaces
$H^2_\balpha(\cU)$ and $H^2_\bbeta(\cY)$ for given non-increasing weight sequences $\balpha$ and $\bbeta$
(let us denote this class by $\cS_{\balpha\to\bbeta}(\cU,\cY)$). The 
analog of Theorem \ref{T:HBSchur} is the 
following.

\begin{theorem}  
\label{T:HBSchur1}
Let the weight sequences $\balpha$ and $\bbeta$ be non-increasing and let $\Psi_\balpha$ and 
$\Psi_\bbeta$ be the associated operator-valued functions defined as in \eqref{2.7}.
The function $S$ is a contractive multiplier from $H^2_\balpha(\cU)$ to $H^2_\bbeta(\cY)$
if and only if  there is an $\bS$  in the Schur class
$\cS(\ell_{2}(\cU), \ell_{2}(\cY))$ so that
 \begin{equation}  \label{Schur-fact1}
     S(z)\Psi_{\balpha}(z) = \Psi_{\bbeta}(z) \bS(z).
   \end{equation}
\end{theorem}

We now consider the interpolation problem \textbf{IP} with 
interpolation condition \eqref{0.12} but now with solution $S$ sought 
in the contractive multiplier class $\cS_{\balpha \to \bbeta}(\cU, 
\cY)$.  One can easily see that the interpolation condition 
\eqref{0.12} can equivalently be expressed in the form
\begin{equation}   \label{alpha-beta-int}
    \cO^{*}_{\bbeta,E,T} M_{S} = \cO^{*}_{\balpha, N, T} \colon 
    H^{2}_{\balpha}(\cU) \to \cX
 \end{equation}
where now we view $M_{S}$ as a multiplication operator from 
$H^{2}_{\balpha}(\cU)$ to $H^{2}_{\bbeta}(\cY)$, i.e., the analogue 
of \eqref{1.1u} holds.  It is then easily seen that the condition
\begin{equation}\label{5.3}
 P: = \cG_{\bbeta,E,T} - \cG_{\balpha,N,T}\ge 0.
  \end{equation}
 is a necessary condition for the existence of a solution $S$ of the 
 interpolation condition \eqref{0.12} in the class $\cS_{\balpha \to 
 \bbeta}(\cU, \cY)$.
 However, unlike the situation for the case 
 $\balpha = {\bf 1}$ where the factorization \eqref{Schur-fact1} can 
 be used to reduce the interpolation problem to a solvable 
 interpolation problem for a classical Schur-class function, the 
 condition \eqref{5.3} in general is not sufficient for the existence 
 of $\cS_{\balpha \to \bbeta}(\cU, \cY)$ solutions, as  
 shown by the following example.
 We note also that the Treil-Volberg result \cite{TV} does not cover this 
 case since the shift $S_{\balpha}$ is a nonisometric contraction, 
 and hence not expansive.

\begin{example}   \label{E:alpha}
 Let $k_n(z,\overline{\zeta})=(1-z\overline{\zeta})^{-n}$ be the reproducing kernel of the standard 
weighted Bergman space $A^2_n$. We want to solve the two-point interpolation problem 
$$
S(\pm1/\sqrt{2})=\pm \sqrt{\frac{26}{15}}
$$
in the class of contractive multipliers from $A^2_2$ to $A^2_3$, i.e., for functions $S$ for which 
the associated kernel 
$$
K_S(z,\overline{\zeta})=k_3(z,\overline{\zeta})-k_2(z,\zeta)S(z)\overline{S(\zeta)}
$$
is positive on $\D\times\D$. The necessary condition \eqref{5.3} is satisfied:
$$
P=\begin{bmatrix}K_S(\frac{1}{\sqrt{2}},\frac{1}{\sqrt{2}}) & K_S(\frac{1}{\sqrt{2}},-\frac{1}{\sqrt{2}})\\
 K_S(-\frac{1}{\sqrt{2}},\frac{1}{\sqrt{2}})&  K_S(-\frac{1}{\sqrt{2}},-\frac{1}{\sqrt{2}}) 
\end{bmatrix}=\frac{16}{15}\cdot\begin{bmatrix} 1& 1\\ 1 & 1\end{bmatrix}\ge 0.
$$
Assume such a function exists. Then the kernel 
$$
\begin{bmatrix}K_S(\frac{1}{\sqrt{2}},\frac{1}{\sqrt{2}}) & K_S(\frac{1}{\sqrt{2}},-\frac{1}{\sqrt{2}}) & 
K_S(\frac{1}{\sqrt{2}},\overline{\zeta}) \\
 K_S(-\frac{1}{\sqrt{2}},\frac{1}{\sqrt{2}})&  K_S(-\frac{1}{\sqrt{2}},-\frac{1}{\sqrt{2}}) &  
K_S(-\frac{1}{\sqrt{2}},\overline{\zeta})\\
K_S(z,\frac{1}{\sqrt{2}}) & K_S(z,-\frac{1}{\sqrt{2}}) & K_S(z,\overline{\zeta})\end{bmatrix}
$$
is positive and, since its $2\times 2$ principal submatrix is singular, we conclude that 
$K_S(z,\frac{1}{\sqrt{2}})\equiv  K_S(z,-\frac{1}{\sqrt{2}})$, i.e., that 
$$
k_3\left(z,\frac{1}{\sqrt{2}}\right)-\sqrt{\frac{26}{15}}k_2\left(z,\frac{1}{\sqrt{2}}\right)S(z)=
k_3\left(z,-\frac{1}{\sqrt{2}}\right)+\sqrt{\frac{26}{15}}k_2\left(z,-\frac{1}{\sqrt{2}}\right)S(z).
$$
Solving the latter equality for $S$ gives
$$
S(z)=\sqrt{\frac{15}{13}}\cdot\frac{z(z^2+6)}{4-z^4}.
$$
Let us show that this $S$ is not a contractive multiplier from $A^2_2$ to $A^2_3$. If it were,
then, since
$$
K_S\left(z,\frac{1}{\sqrt{2}}\right)=\left(1-\frac{z}{\sqrt{2}}\right)^{-3}-\left(1-\frac{z}{\sqrt{2}}\right)^{-2}
\cdot\sqrt{\frac{15}{13}}\cdot\sqrt{\frac{26}{15}}\cdot\frac{z(z^2+6)}{4-z^4}=
\frac{4}{4-z^4},
$$
the kernel 
$$
(z,\overline{\zeta})\mapsto \begin{bmatrix}K_S(\frac{1}{\sqrt{2}},\frac{1}{\sqrt{2}}) & 
K_S(\frac{1}{\sqrt{2}},\overline{\zeta}) \\
K_S(z,\frac{1}{\sqrt{2}}) & K_S(z,\overline{\zeta})\end{bmatrix}=
\begin{bmatrix}\frac{16}{15} & \frac{4}{4-\overline{\zeta}^4}\\
\frac{4}{4-z^4} & K_S(z,\overline{\zeta})\end{bmatrix}
$$
would be positive semidefinite as well as the kernel
$$
\widetilde{K}(z,\overline{\zeta})= K_S(z,\overline{\zeta})-\frac{15}{(4-z^4)(4-\overline{\zeta}^4)}.
$$ 
Therefore, the kernel
$$
\begin{bmatrix}\widetilde{K}(0,0) & \widetilde{K}(0,\overline{\zeta})\\
\widetilde{K}(z,0)& \widetilde{K}(z,\overline{\zeta})\end{bmatrix}=
\begin{bmatrix}\frac{1}{16} & \frac{1-4\overline{\zeta}^4}{4-\overline{\zeta}^4}\\
\frac{1-4z^4}{4-z^4} & \widetilde{K}(z,\overline{\zeta})\end{bmatrix}
$$
was positive, as well as the kernel
$$
\widehat{K}(z,\overline{\zeta})=\widetilde{K}(z,\zeta)-\frac{16(1-4z^4)(1-4\overline{\zeta}^4)}{(4-z^4)(4-\overline{\zeta}^4)}
$$
which is not the case since $\widehat{K}(0.1,0.1)=-0.93276$.
\end{example}

\bibliographystyle{amsplain}

\begin{thebibliography}{10}

%\bibitem{ab} D. Alpay and V. Bolotnikov, {\em  On tangential interpolation in Hilbert
%space modules and some applications}, Oper. Theory Adv. Appl.,
%{\bf 95} (1997), 37--68.

\bibitem{bbieot}
J.~A.~Ball and V.~Bolotnikov, {\it Interpolation problems for Schur
multipliers on the
Drury-Arveson space:  from Nevanlinna-Pick to Abstract Interpolation
Problem}, Integral Equations Operator Theory {\bf 62} (2008), no. 3,
301-349.

\bibitem{bgr} J.~A.~Ball, I.~Gohberg, and L.~Rodman.
\emph{Interpolation of rational matrix functions},
OT45, Birkh\"{a}user Verlag, 1990.

\bibitem{bol}
V.~Bolotnikov,
\emph{Interpolation for multipliers on reproducing kernel 
Hilbert  spaces}, Proc.~Amer.~Math.~Soc., {\bf 131} (2003), 1373-1383.

\bibitem{HKZ} H.~Hedenmalm, B.~Korenblum, and K.~Zhu, {\em Theory of 
Bergman Spaces}, Graduate Texts in Mathematics, Springer, New York, 
2000

\bibitem{kh}
A.~Kheifets, \emph{The abstract interpolation problem and
applications}, \newblock in: \emph{ Holomorphic spaces} (Ed. D.~Sarason, S.~Axler,
J.~McCarthy), Cambridge Univ. Press, Cambridge, 1998, pp.
351--379.

\bibitem{kheyu}
A.~Kheifets and P.~Yuditskii,
\newblock {\em An analysis and extension approach of V. P. Potapov's
approach to scheme interpolation problems with applications to the generalized
bi-tangential {S}chur--{N}evanlinna--{P}ick problem and J-inner-outer
factorization},
in: Matrix and Operator Valued Functions (I. Gohberg and
L.A. Sakhnovich,  eds.), Oper. Theory Adv. Appl., {\bf OT 72},
\newblock Birkh\"auser Verlag, Basel, 1994, pp. 133--161.

\bibitem{olieot}
A.~Olofsson, {\em An operator-valued Berezin transform and the class of 
$n$-hypercontractions}, Integral Equations Operator Theory  {\bf 58} (2007), no. 4, 
503--549. 

\bibitem{olaa} 
A.~Olofsson {\em Operator-valued Bergman inner functions as transfer functions},  
Algebra i Analiz  {\bf 19} (2007),  no. 4, 146--173.

\bibitem{RR} M. Rosenblum and J. Rovnyak,
\emph{Hardy Classes and Operator Theory}, Oxford University Press, 1985.

\bibitem{Rudin}  W.~Rudin,  {\em Principles of Mathematical 
Analysis}, Third Edition, McGraw-Hill, New York, 1976.

\bibitem{Shields} A.~Shields, {\em Weighted shift operators and 
analytic function theory}, in: Mathematical Surveys Volume 
\textbf{13} (Editor C. Pearcy), pp. 49--128, Amer.~Math.~Soc., 
Providence, 1974.

\bibitem{TV}
S. Treil and A. Volberg, {\em A fixed point approach to Nehari's problem and its applications}, 
in: {\em Toeplitz operators and related topics}, Oper. Theory Adv. Appl. {\bf 71}, Birkh\"auser, Basel, 1994, 
pp. 
165--186.

\end{thebibliography}
\providecommand{\bysame}{\leavevmode\hbox to3em{\hrulefill}\thinspace}

\end{document}